\documentclass[reqno]{amsart}

\usepackage{amsmath}
\usepackage{amsfonts}
\usepackage{amssymb}
\usepackage{pb-diagram}
\usepackage{graphicx}
\usepackage[all]{xy}%

\newtheorem{theorem}{Theorem}[section]
\theoremstyle{plain}

\newtheorem{definition}{Definition}
\newtheorem{example}{Example}

\newtheorem{lemma}{Lemma}[section]

\newtheorem{remark}{Remark}

\numberwithin{equation}{section}

\begin{document}
\title[Holomorphic Riemannian maps]{Holomorphic Riemannian maps}
\author{Bayram \c{S}ahin}

\begin{abstract}We introduce holomorphic
Riemannian maps between almost Hermitian manifolds as a generalization of holomorphic submanifolds and holomorphic submersions, give examples and obtain a geometric characterization
of harmonic holomorphic Riemannian maps from almost Hermitian
manifolds to  K\"{a}hler manifolds.
\end{abstract}

\maketitle

\section{Introduction}

\setcounter{equation}{0}
\renewcommand{\theequation}{1.\arabic{equation}}

In 1992, Fischer  introduced Riemannian maps between Riemannian
manifolds in \cite{Fischer} as a generalization of the notions of
isometric immersions and Riemannian submersions, for Riemannian submersions, \cite{Gray} and \cite{O'Neill}, see also \cite{Falcitelli} and \cite{Yano-Kon}. Let $F:(M_1,
g_1)\longrightarrow (M_2, g_2)$  be a smooth map between Riemannian
manifolds such that $0<rank F<min\{ m, n\}$, where $dimM_1=m$ and
$dimM_2=n$. Then we denote the kernel space of $F_*$ by $kerF_*$ and
consider the orthogonal complementary space $\mathcal{H}=(ker
F_*)^\perp$ to $kerF_*$. Then the tangent bundle of $M_1$ has the
following decomposition

$$TM_1=kerF_* \oplus\mathcal{H}.$$

We denote the range of $F_*$ by $rangeF_*$ and consider the
orthogonal complementary space $(range F_*)^\perp$ to $rangeF_*$ in
the tangent bundle $TM_2$ of $M_2$. Since $rankF<min\{ m, n\}$, we
always have $(range F_*)^\perp$. Thus the tangent bundle $TM_2$ of
$M_2$ has the following decomposition
$$TM_2=(rangeF_*)\oplus (rangeF_*)^\perp.$$
 Now, a smooth
map $F:(M^{^m}_1,g_1)\longrightarrow (M^{^n}_2, g_2)$ is called
Riemannian map at $p_1 \in M$ if the horizontal restriction
$F^{^h}_{*p_1}: (ker F_{*p_1})^\perp \longrightarrow (range
F_{*p_1})$  is a linear isometry between the inner product spaces
$((ker F_{*p_1})^\perp, g_1(p_1)\mid_{(ker F_{*p_1})^\perp})$ and
$(range F_{*p_1}, g_2(p_2)\mid_{(range F_{*p_1})})$, $p_2=F(p_1)$.
Therefore Fischer stated in \cite{Fischer} that a Riemannian map is
a map which is as isometric as it can be. In another words, $F_*$
satisfies the equation
\begin{equation}
g_2(F_*X, F_*Y)=g_1(X, Y)\label{eq:1.1}
\end{equation}
 for $X, Y$ vector fields
tangent to $\mathcal{H}$. It follows that isometric immersions and
Riemannian submersions are particular Riemannian maps with
$kerF_*=\{ 0 \}$ and $(range F_*)^\perp=\{ 0 \}$. It is known that a
Riemannian map is a subimmersion \cite{Fischer}.\\

In this paper, as a generalization of holomorphic submersions and holomorphic submanifolds, we
introduce holomorphic Riemannian maps and investigate the harmonicity of such maps.

\section{The Gauss equation for Riemannian maps}
\setcounter{equation}{0}
\renewcommand{\theequation}{2.\arabic{equation}}
Let ($\bar{M}, g$) be an almost Hermitian manifold. This
means \cite{Yano-Kon} that $\bar{M}$ admits a tensor field $J$ of
type (1, 1) on $\bar{M}$ such that, $\forall X, Y \in
\Gamma(T\bar{M})$, we have
\begin{equation}
J^2=-I, \quad g(X, Y)=g(JX, JY) \label{eq:2.12}.
\end{equation}
 An almost Hermitian manifold $\bar{M}$ is called  K\"{a}hler manifold if
\begin{equation}
(\bar{\nabla}_XJ)Y=0,\forall X,Y \in \Gamma(T\bar{M}),
\label{eq:2.13}
\end{equation}
where $\bar{\nabla}$ is the Levi-Civita connection on $\bar{M}. $

Let $(M, g_{_M})$ and $(N, g_{_N})$ be Riemannian manifolds and
suppose that $\varphi: M\longrightarrow N$ is a smooth mapping
between them. Then the differential ${\varphi}_*$ of $\varphi$ can
be viewed a section of the bundle $Hom(TM,
\varphi^{-1}TN)\longrightarrow M,$ where $\varphi^{-1}TN$ is the
pullback bundle which has fibres
$(\varphi^{-1}TN)_p=T_{\varphi(p)} N, p \in M.$ $Hom(TM,
\varphi^{-1}TN)$ has a connection $\nabla$ induced from the
Levi-Civita connection $\nabla^M$ and the pullback connection $\nabla^{\varphi}$.
Then the second fundamental form of $\varphi$ is given by
\begin{equation}
(\nabla {\varphi}_*)(X, Y)=\nabla^{\varphi}_X
{\varphi}_*(Y)-{\varphi}_*(\nabla^M_X Y) \label{eq:2.10}
\end{equation}
for $X, Y \in \Gamma(TM).$ It is known that the second fundamental
form is symmetric. A smooth map $\varphi: (M, g_{_M})
\longrightarrow (N, g_{_N})$ is said to be harmonic if $trace
(\nabla {\varphi}_*)=0.$ On the other hand, the tension field of
$\varphi$ is the section $\tau(\varphi)$ of
$\Gamma(\varphi^{-1}TN)$ defined by
\begin{equation}
\tau(\varphi)=div{\varphi}_*=\sum^{m}_{i=1} (\nabla
{\varphi}_*)(e_i, e_i), \label{eq:2.11}
\end{equation}
where $\{e_1,...,e_m\}$ is the orthonormal frame on $M$. Then it
follows that $\varphi$ is harmonic if and only if
$\tau(\varphi)=0.$\\

For a Riemannian map, we have the following.

\begin{lemma}\cite{Sahin}Let $F$ be a Riemannian map from a
Riemannian manifold $(M_1,g_1)$ to a Riemannian manifold
$(M_2,g_2)$. Then
\begin{equation}
g_2((\nabla F_*)(X,Y),F_*(Z))=0, \forall X,Y,Z \in \Gamma((ker
F_*)^\perp). \label{eq:3.1}
\end{equation}
\end{lemma}

From now on, for simplicity, we denote by $\nabla^2$ both the
Levi-Civita connection of $(M_2, g_2)$ and its pullback along $F$.
 Then according to \cite{Nore}, for any vector field $X$ on $M_1$ and any section $V$ of $(range F_*)^\perp$, where $(range F_*)^\perp$ is the
 subbundle of $F^{-1}(TM_2)$ with fiber $(F_*(T_pM))^\perp$-orthogonal complement of $F_*(T_pM)$ for $g_2$ over $p$, we have
 $\nabla^{^{F \perp}}_XV$ which is the orthogonal projection of $\nabla^2_XV$ on $(F_*(TM))^\perp$. In \cite{Nore}, the author also showed that $\nabla^{^{F \perp}}$ is a linear connection on $(F_*(TM))^\perp$ such that $\nabla^{^{F \perp}}g_2=0$. We now  define $\mathcal{S}_{V}$ as
\begin{equation}
\nabla^2_{_{F_*X}}V=-\mathcal{S}_{_V}F_*X+\nabla^{^{F
\perp}}_{_{X}}V, \label{eq:3.5}
\end{equation}
where $\mathcal{S}_{_V}F_*X$ is the tangential component (a vector
field along $F$) of $\nabla^2_{_{F_*X}}V$. It is easy to see that
$\mathcal{S}_V F_*X$ is bilinear in $V$ and $F_*X$ and
$\mathcal{S}_V F_*X$ at $p$ depends only on $V_p$ and $F_{*p}X_p$.
By direct computations, we obtain
\begin{equation}
g_2(\mathcal{S}_{_V} F_*X,F_*Y)=g_2(V, (\nabla F_*)(X,Y)),
\label{eq:3.6}
\end{equation}
for $X, Y \in \Gamma((ker F_*)^\perp)$ and $V \in \Gamma((range
F_*)^\perp)$. Since $(\nabla F_*)$ is symmetric, it follows that
$\mathcal{S}_{_V}$ is a symmetric linear transformation of $range F_*$.\\

By using (\ref{eq:3.5}) and (\ref{eq:2.10}) we obtain  the following equation which will be called the Gauss equation for a Riemannian map between
Riemannian manifolds.\\

\begin{lemma} Let $F:(M_1,g_1) \longrightarrow
(M_2,g_2)$ be a Riemannian map from Riemannian manifold $M_1$ to a
Riemannian manifold $M_2$. Then we have
\begin{eqnarray}
g_2(R^2(F_*X,F_*Y)F_*Z,F_*T)&=&g_1(R^1(X,Y)Z,T)+g_2((\nabla
F_*)(X,Z),(\nabla F_*)(Y,T))\nonumber\\
&-&g_2((\nabla F_*)(Y,Z),(\nabla F_*)(X,T)) \label{eq:3.7}
\end{eqnarray}
for $X, Y, Z, T \in \Gamma((ker F_*)^\perp)$, where $R^1$ and $R^2$ denote curvature tensors of $\nabla^1$ and $\nabla^2$ which are metric connections on $M_1$ and $M_2$, respectively.
\end{lemma}

\section{Holomorphic Riemannian maps}
\setcounter{equation}{0}
\renewcommand{\theequation}{3.\arabic{equation}}

In this section, we define holomorphic Riemannian maps and obtain a geometric
characterization of harmonic holomorphic Riemannian maps from a
K\"{a}hler manifold to an almost Hermitian manifold.\\

\begin{definition} Let $F$ be a Riemannian map from an
almost Hermitian manifold $(M_1,g_1,J_1)$ to an almost Hermitian
manifold $(M_2,g_2,J_2)$. Then we say that $F$ is a holomorphic
Riemannian map at $p\in M_1$ if
\begin{equation}
J_2F_*=F_*J_1.\label{eq:4.1}
\end{equation}
If $F$ is a holomorphic Riemannian map at every point $p \in M_1$
then we say that $F$ is a holomorphic Riemannian map between $M_1$
and $M_2$.
\end{definition}

It is known that vertical and horizontal distributions of an almost
Hermitian submersion are invariant with respect to the complex
structure of the total manifold. Next, we show that this is true for
a holomorphic Riemannian map.\\

\begin{lemma}Let $F$ be a holomorphic Riemannian
map between almost Hermitian manifolds $(M_1,g_1,J_1)$ and
$(M_2,g_2,J_2)$. Then the distributions $ker F_*$ and $(ker
F_*)^\perp$ are invariant with respect to $J_1$.
\end{lemma}

\begin{proof}For $X \in \Gamma(ker F_*)$, from
(\ref{eq:4.1}) we have  $F_*(J_1X)=J_2F_*(X)=0$ which implies that
$J_1X\in \Gamma(ker F_*)$. In a similar way, one shows that $(ker F_*)^\perp$ is invariant.
\end{proof}

In a similar way, it is easy to see that $(range F_*)^\perp$ is
invariant under the action of $J_2$. We now give examples of holomorphic Riemannian maps.\\

\begin{example}Every holomorphic submersion between
almost Hermitian manifolds is a holomorphic Riemannian map with
$(range F_*)^\perp=\{0\}$, For holomorphic (almost Hermitian) submersions, see; \cite{Falcitelli}, \cite{Watson}.
\end{example}

\begin{example} Every K\"{a}hlerian submanifold of a
K\"{a}hler manifold is a holomorphic Riemannian map with $ker
F_*=\{0\}$. For K\"{a}hlerian submanifolds, see; \cite{Yano-Kon}.
\end{example}
In the following $R^{2m}$ denotes the Euclidean $2m-$ space with the
standart metric. An almost complex structure $J$ on $R^{2m}$ is said
to be compatible if $(R^{2m},J)$ is complex analytically isometric
to the complex number space $C^m$ with the standart flat
K\"{a}hlerian metric. We denote by $J$ the compatible almost complex
structure on $R^{2m}$ defined by
$$J(a^1,...,a^{2m})=(-a^2,a^1,...,-a^{2m},a^{2m-1}).$$
\begin{example}Consider the following Riemannian map given by
$$
\begin{array}{cccc}
  F: & R^4             & \longrightarrow & R^4\\
     & (x_1,x_2,x_3,x_4) &             & (\frac{x_1 + x_3}{\sqrt{2}},
     \frac{x_2+x_4}{\sqrt{2}},0,0).
\end{array}
$$
Then $F$ is a holomorphic Riemannian map.
\end{example}

\begin{remark}We note that the notion of invariant
Riemannian map has been introduced in \cite{Sahin} as a
generalization of invariant immersion of almost Hermitian manifolds
and holomorphic Riemannian submersions. One can see that every
holomorphic Riemannian map is an invariant Riemannian map, but the
converse is not true. In other words, an invariant Riemannian map
may not be a holomorphic Riemannian map.
\end{remark}

Since $F$ is a subimmersion, it follows that the rank of $F$ is constant on $M_1$,
then the rank theorem for functions implies that $kerF_*$ is an integrable subbundle
of $TM_1$, (\cite{Marsden}, page:205). We now investigate the harmonicity of holomorphic Riemannian maps. We
first note that if $M_1$ and $M_2$ are K\"{a}hler manifolds and $F : M_1\longrightarrow M_2$ is a
holomorphic map then $F$ is harmonic \cite{Wood-Baird}. But there is no guarantee when $M_1$
or $M_2 $ is an almost Hermitian manifold.\\

\begin{theorem}Let $F$ be a holomorphic Riemannian
map from a K\"{a}hler manifold $(M_1,g_,J_1)$ to almost Hermitian
manifold $(M_2,g_2,J_2)$. Then $F$ is harmonic if and only if the
distribution $F_*((ker F_*)^\perp)$ is minimal.
\end{theorem}

\begin{proof} Since $TM_1=ker F_*\oplus (ker F_*)^\perp$,
we can write $\tau=\tau^1+\tau^2$, where $\tau^1$ and $\tau^2$ are
the parts of $\tau$ in $ker F_*$ and $(ker F_*)^\perp$,
respectively. First we compute $\tau^1=\sum^{n_1}_{i=1}(\nabla
F_*)(e_i,e_i)$, where $\{e_1,...,e_{n_1}\}$ is a basis of $ker F_*$.
From (\ref{eq:2.10}), we have
\begin{equation}
\tau^1=-\sum^{n_1}_{i=1}F_*(\nabla^1_{e_i}e_i).\label{eq:4.2}
\end{equation}
We note that, since $(ker F_*)$ is an invariant space with respect
to $J_1$, then $\{J_1e_i\}^{n_1}_{i=1}$ is also basis of $ker F_*$.
Thus we can write $$\tau^1=\sum^{n_1}_{i=1}(\nabla
F_*)(J_1e_i,J_1e_i)=-\sum^{n_1}_{i=1} F_*(\nabla^1_{J_1e_i}J_1e_i).$$
Since $M_1$ is a K\"{a}hler manifold and $ker F_*$ is integrable, using (\ref{eq:4.1}), we
obtain
$$\tau^1=-\sum^{n_1}_{i=1}J_2F_*(\nabla^1_{e_i}J_1e_i).$$
Using again (\ref{eq:4.1}), we derive
\begin{equation}
\tau^1=\sum^{n_1}_{i=1}F_*(\nabla^1_{e_i}e_i).\label{eq:4.3}
\end{equation} Thus (\ref{eq:4.2}) and (\ref{eq:4.3}) imply that
$\tau^1=0$. On the other hand, using Lemma~2.1 and (\ref{eq:2.10}) we obtain
$$
\tau^2=g_2(\sum^{m_2}_{s=1}\sum^{n_1}_{a=1}(\nabla^F_{e_a}F_*(e_a),\mu_s)\mu_s=H_{(range F_*)},
$$
where $H_{(range F_*)}$ is the mean curvature vector field of
$(range F_*)$. Then our assertion follows from above equation and
(\ref{eq:4.3}).
\end{proof}

Next, by using (\ref{eq:2.12}) and (\ref{eq:2.13}) we have the
following.

\begin{lemma} Let $F$ be a holomorphic Riemannian
map from an almost Hermitian  manifold $(M_1,g_1,J_1)$ to a
K\"{a}hler manifold $(M_2,g_2,J_2)$. Then we have
\begin{equation}
(\nabla F_*)(X,J_1Y)=(\nabla F_*)( Y,J_1X)=J_2(\nabla F_*)(X,Y),
\label{eq:4.4}
\end{equation}
for $X,Y \in \Gamma((ker F_*)^\perp)$.
\end{lemma}

\begin{lemma} Let $F$ be a holomorphic Riemannian
map from an almost Hermitian  manifold $(M_1,g_1,J_1)$ to a
K\"{a}hler manifold $(M_2,g_2,J_2)$. Then we have
\begin{eqnarray}
g_1(R^1(X,J_1X)J_1X,X)&=&g_2(R^2(F_*X,J_2F_*X)J_2F_*X,F_*X)\nonumber\\
&-&2\parallel (\nabla F_*)(X,X)\parallel^2 \label{eq:4.5}
\end{eqnarray}
for $X \in \Gamma((ker F_*)^\perp)$.
\end{lemma}

\begin{proof} Putting $Y=J_1X$, $Z=J_1X$ and $T=X$ in
(\ref{eq:3.7}) and by using (\ref{eq:4.1}) and (\ref{eq:4.4}) we obtain (\ref{eq:4.5}).
\end{proof}

As a result of Lemma ~3.3, we have the following result for the
leaves of $(ker F_*)^\perp$.

\begin{theorem}Let $F$ be a holomorphic Riemannian
map from an almost Hermitian manifold $(M_1,g_1,J_1)$ to a complex
space form $(M_2(c),g_2,J_2)$ of constant holomorphic sectional
curvature $c$ such that $(ker F_*)^\perp$ is integrable. Then the
integral manifold of $(ker F_*)^\perp$ is a complex space form
$M'(c)$ if and only if $(\nabla F_*)(X,X)=0$ for $X \in \Gamma((ker
F_*)^\perp)$.
\end{theorem}
\noindent{\bf Concluding Remarks.~}It is known that the complex techniques in relativity have been very effective tools for understanding spacetime geometry \cite{Lerner}. Indeed, complex manifolds have two interesting classes of K\"{a}hler manifolds. One is Calabi-Yau manifolds which have their applications in superstring theory \cite{Candelas}. The other one is Teichmuler spaces applicable to relativity \cite{Tromba}. It is also important to note that CR-structures have been extensively used in spacetime geometry of relativity \cite{Penrose}. For complex methods in general relativity,  see:\cite{Esposito}.\\

In \cite{Fischer}, Fischer proposed an approach to build a quantum model and he pointed
out the success of such a program of building a quantum model of nature using
Riemannian maps would provide an interesting relationship between Riemannian
maps, harmonic maps and Lagrangian field theory on the mathematical side, and
Maxwell's equation, Schr\"{o}dinger's equation and their proposed generalization on
the physical side. It is also important to
note that Riemannian maps satisfy the Eikonal equation which is a bridge between
geometric optics and physical optics. For Riemannian maps and their applications
in spacetime geometry, see \cite{Garcia-Rio-Kupeli}. As a unification of Riemannian maps and complex geometry, holomorphic Riemannian maps may have their applications in mathematical physics and physical optics.

\end{document}